\documentclass[11pt, leqno]{article}
\usepackage{amsmath,amssymb,amsthm, epsfig}

\usepackage{hyperref}

\usepackage{amsmath}

\usepackage{amssymb}

\usepackage{color}

\usepackage{ulem}

\usepackage{epsfig}

\usepackage[mathscr]{eucal}

\usepackage{stackrel}

\usepackage[nottoc,notlot,notlof]{tocbibind}

\usepackage[mathscr]{eucal}

\usepackage{stackrel}

\usepackage{mathptmx} % rm & math
\usepackage[scaled=0.90]{helvet} % ss
\usepackage{courier} % tt
\normalfont
\usepackage[T1]{fontenc}

\title{Geometric regularity estimates for fully nonlinear elliptic equations with free boundaries}

\author{\it by \smallskip \\
 J.V. da Silva\footnote{\noindent \textsc{Jo\~{a}o Vitor da Silva}. Departamento de Matem\'{a}tica - Instituto de Ci\^{e}ncias Exatas - Universidade de Bras\'{i}lia - UnB, Campus Universit\'{a}rio Darcy Ribeiro, 70910-900, Bras\'{i}lia, DF, Brazil. \noindent \texttt{E-mail: \url{J.V.Silva@mat.unb.br}}} \footnote{\noindent \textsc{Jo\~{a}o Vitor da Silva}. Instituto de Investigaciones Matem\'{a}ticas Luis A. Santal\'{o} (IMAS) - CONICET (Argentine), Ciudad Universitaria, Pabell\'{o}n I (1428) Av. Cantilo s/n - Buenos Aires \texttt{E-mail: \url{jdasilva@dm.uba.ar}}}
, \qquad  R.A. Leit\~{a}o\footnote{\noindent \textsc{Raimundo Alves Leit\~{a}o J\'{u}nior}.
Universidade Federal Cear\'{a} - UFC. Department of Mathematics. Fortaleza - CE, Brazil - 60455-760.
\texttt{E-mail: \url{rleitao@mat.ufc.br}}}\qquad $\&$ \qquad G.C. Ricarte\footnote{\noindent \textsc{Gleydson Chaves Ricarte}.
Universidade Federal Cear\'{a} - UFC. Department of Mathematics. Fortaleza - CE, Brazil - 60455-760. \texttt{E-mail: \url{ricarte@mat.ufc.br}}
}}

\date{}
\newlength{\hchng}
\newlength{\vchng}
\def \R {\mathbb{R}}

\def \dist {\mathrm{dist}}

\def \Leb {\mathscr{L}^N}
\def \tr {\mathrm{Tr}}

\newcommand{\defeq}{\mathrel{\mathop:}=}

\newtheorem{theorem}{Theorem}[section]

\newtheorem{lemma}[theorem]{Lemma}

\newtheorem{corollary}[theorem]{Corollary}

\theoremstyle{definition}

\newtheorem{definition}[theorem]{Definition}

\newtheorem{example}[theorem]{Example}

\theoremstyle{remark}

\newtheorem{remark}[theorem]{Remark}

\numberwithin{equation}{section}

\newcommand{\intav}[1]{\mathchoice {\mathop{\vrule width 6pt height 3 pt depth  -2.5pt
\kern -8pt \intop}\nolimits_{\kern -6pt#1}} {\mathop{\vrule width
5pt height 3  pt depth -2.6pt \kern -6pt \intop}\nolimits_{#1}}
{\mathop{\vrule width 5pt height 3 pt depth -2.6pt \kern -6pt
\intop}\nolimits_{#1}} {\mathop{\vrule width 5pt height 3 pt depth
-2.6pt \kern -6pt \intop}\nolimits_{#1}}}

\begin{document}
\maketitle

\begin{abstract}
In this manuscript we study geometric regularity estimates for problems driven by fully nonlinear elliptic operators (which can be either degenerate or singular when ``the gradient is small'') under strong absorption conditions of the general form:
\begin{equation}\label{eq1}
      \mathrm{G}(x, D u, D^2 u) = f(u)\chi_{\{u>0\}} \quad \mbox{in} \quad \Omega,
\end{equation}
where the mapping $u \mapsto f(u)$ fails to decrease fast enough at the origin, so allowing that non-negative solutions may create plateau regions, that is, \textit{a priori} unknown subsets where a given solution vanishes identically. We establish improved geometric $C_{\text{loc}}^{\kappa}$ regularity along the set $\mathrm{F}_0 = \partial \{u>0\} \cap \Omega$ (the free boundary of the model), for a sharp value of $\displaystyle \kappa \gg 1$ (obtained explicitly) depending only on structural parameters. Non-degeneracy among others measure theoretical properties are also obtained. A sharp Liouville result for entire solutions with controlled growth at infinity is proved. We also present a number of applications consequential of our findings.
%\tableofcontents
\end{abstract}

\newpage

\section{Introduction}

\subsection{Statement of main results}

\hspace{0.6cm}In this manuscript we study regularity issues for reaction-diffusion problems governed by second order nonlinear elliptic equations (possibly of degenerate or singular type) for which a Minimum Principle is not available (see \cite{BdaL} for an essay on this topic):
\begin{equation}\label{DCP1}
\left\{
\begin{array}{rclcc}
  F(x, Du, D^2u) + |Du|^{\gamma}\langle\mathfrak{b}(x), Du\rangle & = & f(u)\chi_{\{u>0\}} & \mbox{in} & \Omega \\
  u(x) & = & g(x) & \mbox{on} & \partial \Omega,
\end{array}
\right.
\end{equation}
where $\Omega \subset \R^N$ is a smooth, open and bounded domain, $0\leq g \in C^0(\partial \Omega)$, $\mathfrak{b} \in C^0(\overline{\Omega}, \R^N)$, $f$ is a continuous and increasing function with $f(0) = 0$ and $F: \Omega \times (\R^N \backslash \{0\}) \times Sym(N) \to \R$ is a second order fully nonlinear elliptic operator with measurable coefficients satisfying:

\begin{enumerate}
\item[{\bf(F1)}]\label{F1}[{\bf $(\lambda, \Lambda, \gamma)$-Ellipticity condition}] There exist constants $ \Lambda \geq \lambda >0$ such that for any $(x,\overrightarrow{p}) \in \Omega \times (\R^N\setminus\{0\})$ and $M, P \in\text{Sym}(N) $, with $P\ge 0$ and $\gamma >-1$ there holds
$$
|\overrightarrow{p}|^{\gamma}\mathcal{P}_{\lambda, \Lambda}^{-}(P)\leq F(x, \overrightarrow{p}, M+P)-F(x,  \overrightarrow{p}, M)\leq  |\overrightarrow{p}|^{\gamma}\mathcal{P}_{\lambda, \Lambda}^{+}(P),
$$
where $\mathcal{P}^{+}_{\lambda,\Lambda}$ and $\mathcal{P}^{-}_{\lambda,\Lambda}$ denote the \textit{Pucci's extremal operators}:
$$
     \mathcal{P}^{+}_{\lambda,\Lambda}(M) := \lambda \cdot \sum_{e_i <0} e_i  + \Lambda \cdot \sum_{e_i >0} e_i \quad \mbox{and} \quad \mathcal{P}^{-}_{\lambda,\Lambda}(M) := \lambda \cdot \sum_{e_i >0} e_i  + \Lambda \cdot \sum_{e_i <0} e_i
$$
and $\{e_i : 1 \le i \le N\}$ are the eigenvalues of $M$.

\item[{\bf(F2)}]\label{F2}[{\bf $\gamma$-Homogeneity condition}] For all $s \in \R^{\ast}$ and $(x,\overrightarrow{p}, M) \in \Omega \times (\R^n\setminus\{0\}) \times Sym(N)$,
$$
F(x, s\overrightarrow{p}, M) = |s|^{\gamma} F(x, \overrightarrow{p}, M),
$$
\end{enumerate}
We recommend reading Ara\'{u}jo \textit{et al}, Birindelli-Demengel and Imbert-Silvestre's works \cite{ART15}, \cite{BerDem}, \cite{BeDe1}, \cite{BD2}, \cite{BD3} and \cite{IS} for a number of examples of operators with such structural properties.

Motivated by several issues coming from pure and applied sciences, we will focus our attention, particularly on prototypes driven by a strong absorption condition, whose standard model is given by
\begin{equation}\label{Maineq}
    \mathrm{G}(x, Du, D^2u) =  F(x, Du, D^2u) + |Du|^{\gamma}\langle\mathfrak{b}(x), Du\rangle = \lambda_0(x).u^{\mu}\chi_{\{u>0\}}(x) \quad \mbox{in} \quad \Omega,
\end{equation}
where $0\leq \mu < \gamma+1$ is the \textit{absorption factor} and $\lambda_0 \in C^0(\overline{\Omega})$ (\textit{Thiele modulus} - bounded away from zero and infinity). Such models are mathematically interesting because they enable the formation of \textit{dead-core} sets, i.e. \textit{a priori} unknown regions where non-negative solutions vanish identically. By way of contrast, the study of \eqref{Maineq} is relevant, not only for its several applications, but especially for its innate relation with a number of relevant free boundary problems in the literature (see \textit{e.g.} \cite{AP}, \cite{ADD16}, \cite{Aris1}, \cite{Aris2}, \cite{DD15}, \cite[Chapter 1]{Diaz}, \cite{FriePhil}, \cite{LeeShah} and \cite{Phil} for some variational examples and \cite{CafSal}, \cite{CafSalShah}, \cite{Tei14} and \cite{Tei18} for a non-variational counterpart). For this reason, understanding the ``intrinsic geometry'' of the former model is an important step in comprising the behavior of dead-core solutions near their free boundary points. We will provide a brief discussion of this soon.

In our first result, we establish the precise asymptotic behavior at which non-negative viscosity solutions leave their dead-core sets. This is an important piece of information in several free boundary problems (see \cite{ALT}, \cite{ART15}, \cite{LeeShah}, \cite{RT} and \cite{Tei4}), and it plays a pivotal role in establishing many weak geometric properties (see Section \ref{Cons} for more details).

\begin{theorem}[{\bf Non-degeneracy}]\label{LGR} Let $u$ be a nonnegative, bounded viscosity solution to \eqref{Maineq} in $B_1$ and let $x_0 \in \overline{\{u >0\}} \cap B_{\frac{1}{2}}$ be a point in the closure of the non-coincidence set. Then for any $0<r<\frac{1}{2}$, there holds
$$
   \displaystyle \sup_{B_r(x_0)} \,u(x) \geq C_{\sharp}\left(N, \Lambda, \inf_{\Omega} \lambda_0(x), \gamma, \mu\right).r^{\frac{\gamma+2}{\gamma+1-\mu}}.
$$
\end{theorem}

In our next result we establish a sharp and improved regularity estimate at free boundary points.

\begin{theorem}[{\bf Improved regularity along free boundary}]\label{IRThm} Let $u$ be a nonnegative and bounded viscosity solution to \eqref{Maineq} and consider $Y_0 \in \partial\{u > 0\} \cap \Omega^{\prime}$ a free boundary point with $\Omega^{\prime} \Subset \Omega$. Then for $r_0 \ll  \min\left\{1, \frac{\dist(\Omega^{\prime}, \partial \Omega)}{2}\right\}$ and any $X \in B_{r_0}(Y_0) \cap \{u > 0\}$ there holds
$$
   u(X)\leq C^{\sharp}.\max\left\{1, \|u\|_{L^\infty(\Omega)}, \|\lambda_0\|_{L^{\infty}(\Omega)}^{\frac{1}{\gamma+1-\mu}}\right\}.|X-Y_0|^{\frac{\gamma+2}{\gamma+1-\mu}},
$$
where $C^{\sharp}>0$ depends only on $\displaystyle N, \lambda, \Lambda, \gamma, \mu, \|\mathfrak{b}\|_{L^{\infty}(\Omega)}, \|\lambda_0\|_{L^{\infty}(\Omega)}$ and $\dist(\Omega^{\prime}, \partial \Omega)$.
\end{theorem}

 To the best of our knowledge and as far as degenerate/singular elliptic models in non-divergence form are concerned, Theorem \ref{IRThm} is a novelty even for dead-core problems as follows
$$
  |Du|^{\gamma}\left(\tr(\mathrm{A}(x)D^2 u) + \langle\mathfrak{b}(x), Du\rangle\right)= \lambda_0(x)u^{\mu}\chi_{\{u>0\}} \,\,\, \text{for} \,\,\, 0 \neq \gamma >-1.
$$

 The insights behind the proof of Theorem \ref{IRThm} one resume to two important steps: a (finer) asymptotic blow-up analysis and the use of a sharp \textit{a priori} bound device (via Harnack type inequality). Heuristically, for an appropriate family of normalized and scaled solutions, if the ``magnitude'' of the \textit{Thiele modulus} is under control (with small enough bound), then we must expect the following iterative geometric decay (for a universal constant $C_0>0$)
 $$
    \displaystyle \sup_{B_{r_0^k}(x_0)} u(x) \leq C_0 r_0^{k\left(\frac{\gamma+2}{\gamma+1-\mu}\right)}
    \,\, \forall \,\,x_0 \in \partial\{u>0\}\cap \Omega^{\prime}, \,\,r_0 \ll 1, \text{and} \,\,k \in \mathbb{N},
 $$
which will yield the desired geometric regularity estimate along free boundary points (see Section \ref{ImpRegEst}).

We are also able to provide a crystal clear answer as to what will happen with solutions of \eqref{Maineq} when $\mu = \gamma+1$. It is worthwhile to mention that in such a scenario, the former regularity estimates ``collapse''. For this reason, analyzing such a borderline setting is a delicate and challenging task.

\begin{theorem}\label{UCPthm} Let $u$ be a non-negative, bounded viscosity solution to \eqref{Maineq} with $\mu = \gamma+1$. Then, the following dichotomy holds: either $u>0$ or $u \equiv 0$ in $\Omega$.
\end{theorem}

We also find the sharp (and improved) rate, which gradient decays at interior free boundary points.

\begin{theorem}[{\bf Sharp gradient decay}]\label{IRresult2} Let $u$ be a bounded non-negative viscosity solution to \eqref{Maineq}. Then, for any point $z \in \partial \{u > 0\} \cap \Omega^{\prime}$ for $\Omega^{\prime} \Subset \Omega$, there exists a universal constant $C>0$ such that
$$
   \displaystyle \sup_{B_r(z)} |D u(x)| \leq  Cr^{\frac{1+\mu}{\gamma+1-\mu}} \quad \text{for all} \quad  0<r \ll \min\left\{1, \frac{\dist(\Omega^{\prime}, \partial \Omega)}{2}\right\}.
$$
\end{theorem}

In the following, we present a sharp Liouville type result, which holds provided that entire solutions satisfy a controlled growth condition at infinity (cf. \cite[Theorem 3.1]{BerDem}).

\begin{theorem}[{\bf Liouville}]\label{Liouville} Let $u$ be a non-negative viscosity solution to \eqref{Maineq} in $\R^N$. Then, $u\equiv 0$ provided that
\begin{equation}\label{cond thm B1}
\displaystyle \limsup_{|x| \to \infty} \frac{u(x)}{|x|^{\frac{\gamma+2}{\gamma+1-\mu}}} < \left[\inf_{\R^N} \lambda_0(x) \frac{ \left(\gamma+1 - \mu \right)^{\gamma + 2} }{N \Lambda\left(\mu+1 \right) \left( \gamma+2\right)^{\gamma + 1}}\right]^{\frac{1}{\gamma+1 - \mu}}.
\end{equation}
\end{theorem}

\subsection{Motivation and first insights into the theory}

  \hspace{0.6cm} Throughout the last decades a number of reaction-diffusion elliptic equations have
  emerged as models of many phenomena coming from pure and applied sciences, where some remarkable examples appear in chemical reactions, physical-mathematical phenomena, biological processes and population dynamics just to name a few. Notwithstanding, reaction-diffusion processes with one-phase transition, i.e. with sign constraint, are often more interesting from the applied point of view, since they constitute the only significant (realistic) situation. An illustrative example coming from certain (stationary) isothermal, and irreversible catalytic reaction-diffusion process (see \textit{e.g.} \cite{Aris1}, \cite{Aris2} and \cite[Chapter 1]{Diaz} for remarkable surveys) is
\begin{equation}\label{DCP}
\left\{
\begin{array}{rclcl}
     -\Delta u(x) + \lambda_0.f(u)\chi_{\{u>0\}}& = & 0 & \mbox{in} & \Omega \\
     u(x) & = & 1 & \mbox{on} & \partial \Omega,
\end{array}
\right.
\end{equation}
  where $\Omega \subset \R^N$ is a regular and bounded domain, $f: \R_{+} \to \R_{+}$ is a continuous and increasing reaction term, with $f(0) \geq 0$. For such a model, $u$ stands the concentration of a certain chemical reagent (or gas) under a prescribed isothermal flow on the boundary. Furthermore, $f(u)$ represents the ratio of reaction rate at concentration $u$ to reaction rate at concentration unity, and $\lambda_0>0$ (the \textit{Thiele modulus}) controls the ratio of reaction rate to diffusion-convection rate.

 Recall that when $f \in C^{0,1}(\R_{+})$, it follows from the Maximum Principle that nonnegative solutions of \eqref{DCP} must be strictly positive. However, if $f$ fails to be Lipschitz (or even not decays fast enough at the origin), \textit{e.g.} as $f(t) = t^q$ with $q \in [0, 1)$, then such solutions may exhibit plateaus zones, also known as \textit{dead core} sets, i.e. regions $\Omega_0 \Subset \Omega$ where solutions vanish identically. From a physical-mathematical point of view, such a phenomenon revels that where such a solution is vanishing, it delineates a region (\textit{a priori} unknown) where no diffusion process is present (i.e. the chemical substance is ``wasted''). Such a feature enable us to treat \eqref{DCP} as a problem with free boundaries.

  Now, let us move towards the theory of non-divergence form elliptic equations of degenerate/singular type. Recall that solutions to such PDEs in general have a lack of good regularity estimates (\textit{e.g.} $p$-Laplacian operator). Indeed, several mathematical models involving such operators have their ``ellipticity factor'' collapsing along an unknown set, the \textit{free boundary}. Particularly, this phenomenon implies less diffusivity for the model near such a set, and consequently regularity properties of solutions one become, not quite so simple to be established. A typical case of such a phenomenon occurs as the governing operator is anisotropic and the ``modulus of ellipticity'' degenerates along the set of critical points of solutions, i.e.
$$
     \mathcal{C}_{\Omega}(u) \defeq \left\{x \in \Omega: Du(x)=\overrightarrow{0}\right\}.
$$
An interesting and simple example is given by the model
$$
    \mathrm{G}(x, Du, D^2 u) = \mathfrak{g}(x, |Du|)\tr(\mathrm{A}(x)D^2 u) = \mathfrak{f}(u) \quad \text{in} \quad \Omega,
$$
where $\mathfrak{f}: \R \to [0, \infty)$ and $\mathfrak{g}:\Omega \times [0, \infty) \to [0, \infty]$ are continuous functions with $g(x, 0) = 0$ (or $g(x, 0) = +\infty$), $\mathrm{A} \in C^0(\overline{\Omega}, \R^{2N})$ (uniformly elliptic matrix) and $u \mapsto \mathfrak{f}(u)$ is non-decreasing with $\mathfrak{f}(s) = 0$ for every $s\leq 0$. Notice that solutions for such a model might only be ``irregular'' if $|Du|$ is small or large enough. So the delicate question here consists of understanding the precise (local) behavior of solutions close theirs two free boundaries, namely the physical free boundary $\partial\{u>0\}$ and the ``non-physical'' one $\partial \{|Du| > 0\}$ when such a model presents ``absorption factor'' ($f(t) \approx t_{+}^{\mu}$) lesser than the homogeneity degree of governing operator (\textit{e.g.} as $g(x, s)\approx s^{\gamma}$), so enabling the absence of strong maximum principle (compare with \cite{BdaL} for the validity of the strong maximum principle). Therefore, the understanding of such an interplay is essential in order to study (finer) regularity issues to \eqref{Maineq}.

 Despite of the fact that there is a huge amount of literature on dead-core problems in divergence form and theirs qualitative features, quantitative counterparts in non-variational elliptic models like \eqref{Maineq} are far less studied due to the rigidity of the structure of such operators (see \cite{ALT}, \cite{daSO} and \cite{Tei4} as enlightening examples). Therefore, the treatment of such free boundary problems requires the development of new ideas and modern techniques. This lack of investigation has been our main impetus in studying fully nonlinear models with non-uniformly elliptic (anisotropic) structure under strong absorption conditions, which focus on a modern, systematic and non-variational approach for a general class of problems with free boundaries.

\subsection{An overview on our results and their connections with other theories}

\hspace{0.6cm}Our analysis relies strongly on methods of the so-termed \textit{Geometric Regularity Theory}: an approach to regularity theory was born in the fully nonlinear setting, with the Caffarelli's seminal work (see \cite{C89} and  \cite{CC95}). Currently, Teixeira's outstanding survey \cite{Tei2} summarizes the  \textit{state of the art} for such techniques and methods coming from Geometric Measure Theory, Harmonic Analysis, Nonlinear Analysis and PDEs, Nonlinear Potential Theory, and Free Boundary Problems (see also \cite{AdaSRT18}, \cite{ART15}, \cite{daSDosP18}, \cite{daST17}, \cite{Tei14}, \cite{T15} and \cite{Tei18} for further current examples this subject).

 It is worth highlighting that our results generalize and/or recover and/or complement previous ones (see \cite{Tei4} (dead-core problems with uniformly elliptic structure), \cite[Section 4]{Tei2} ($p-$dead core problems) and \cite{Tei18} (\textit{quenching type} problems), see also \cite{ALT}) by allowing deal with more general classes of operators. Moreover, different from the former results, we suggest different approaches and techniques to address such regularity issues, which enable us to establish a number of significant applications for our findings.

 Now, let us connect our results with certain equations from the theory of superconductivity. To that end, consider the simple prototype $G(x, Du, D^2 u) = |Du|^{\gamma} F(x, D^2 u)$, where $\gamma>0$ and $F: \Omega \times \text{Sym}(N) \to \R$ is a $(\lambda, \Lambda)-$elliptic operator. In face of \eqref{Maineq} we observe that
$$
   \{|Du|=0\} \cap \Omega \subset \{u = 0\} \cap \Omega.
$$
Hence, instead of \eqref{DCP1} we may consider the most realistic accuracy (physical-mathematical) model
\begin{equation}\label{EqSUpercond}
     |Du|^{\gamma} F(x, D^2 u) = f(u)\chi_{\{|Du|>0\}}(x),
\end{equation}
which can be naturally thought, to some extent, as a more general prototype (with the adding of a degeneracy factor) in the theory of superconductivity (i.e. as $\gamma = 0$), where fully nonlinear equations with patches of zero gradient drives the mathematical model. Recall (see \cite{CafSal}, \cite{CafSalShah} and \cite{Chap}) that \eqref{EqSUpercond} (for $\gamma = 0$) stands the stationary situation for the mean field theory of superconductivity (vortices), provided the scalar stream function admits a functional dependence on the scalar magnetic potential. In this scenario, it was proved (see \cite[Corollary 7]{CafSal}) that viscosity solutions to \eqref{EqSUpercond} are merely $C^{0, \alpha}$ for some $\alpha \in (0, 1)$. Moreover, under the concavity assumption on $F$ was proved (see \cite[Corollary 8]{CafSal}) that solutions are in $W^{2, p}$. Nevertheless, even for $\gamma = 0$, our results (see Theorem \ref{IRThm}) are striking, because in this setting, we obtain $C^{\frac{2}{1-\mu}}$ local regularity along free boundary points, which represents an unexpected gain of smoothness along gradient degenerate points (compare also with \cite[Theorem 4.1]{ART15} and \cite[Theorem 1 and Theorem 3]{Tei14}).

 Thanks to Theorem \ref{IRThm} (resp. Theorem  \ref{IRresult2}) we are able to access better regularity estimates (at free boundary points) than those currently available. As a matter of fact, in our approach we impose just continuous coefficients for the governing operator $F$. Nevertheless, the modulus of continuity improves upon the expected H\"{o}lder regularity coming from Krylov-Safonov type estimates (see \cite{DFQ2} and \cite{IS16}). Furthermore, even for constant coefficient problems, $F(Du, D^2 u)= f(x, u) \in L^{\infty}$, our result is remarkable, because in this scenario $C_{\text{loc}}^{1, \min\left\{\alpha_{\mathrm{H}}^{-}, \frac{1}{1+\gamma}\right\}}$-estimates are the best expected regularity (see \cite{ART15}, \cite{BD2}, \cite{BD3}, \cite{C89}, \cite{CC95} and \cite{IS}). By way of illustration, taking into account the model $\mathrm{G}(x, Du, D^2 u) = |Du|^{\gamma} F(x, D^2 u)$, with $F$ a concave/convex operator, it was proved in \cite[Corollary 3.2]{ART15} sharp $C_{\text{loc}}^{1, \frac{1}{\gamma+1}}$ regularity estimates for degenerate elliptic equations. Nevertheless, Theorem \ref{IRThm} (resp. Theorem \ref{IRresult2}) presents a sharp/improved modulus of continuity (at free boundary points), i.e.
$$
   \kappa(\gamma, \mu) = \frac{\gamma+2}{\gamma+1-\mu} \geq 1+\frac{1}{\gamma+1} \quad \left(\text{resp.}\,\,\frac{1+\mu}{\gamma+1-\mu} \geq \frac{1}{\gamma+1}\right) \,\,\,\,(\text{sharp and improved  exponents}).
$$

Finally, our paper is organized as follows: in Section \ref{Section2} we present an appropriate notion of viscosity solutions to our context. Yet in Section \ref{Section2}, we deliver few results about \textit{a priori} estimates, comparison and existence of dead core solutions. Section \ref{Geommeasprop} is devoted to proving some sharp non-degeneracy results, in particular Theorem \ref{LGR}. In Section \ref{ImpRegEst}, we  prove a central result, namely Lemma \ref{FlatLemma}, which allows us to place solutions in a flatness improvement regime. Yet in Section \ref{ImpRegEst} we deliver a proof of Theorem \ref{IRThm}. At the end of Section \ref{ImpRegEst}, we analyze the borderline case, i.e., Theorem \ref{UCPthm}. In Section \ref{Cons} is dedicated to  applications of the our main results. Moreover, some weak geometric properties such as uniform positive density and porosity of the free boundary are established. In Section \ref{Sec6} we prove our Liouville type result, i.e. Theorem \ref{Liouville}.

\section{Background results}\label{Section2}

Let us review the definition of viscosity solution for our operators. For $\mathrm{G}: \Omega \times (\R^N \setminus \{0\}) \times Sym(N) \to \R$ (fulfilling (F1)-(F2)) and $f: \Omega \times \R \to \R$ continuous functions we have the following:

\begin{definition}[{\bf Viscosity solutions}] $u \in C^{0}(\Omega)$ is a viscosity super-solution (resp. sub-solution) to
$$
    \mathrm{G}(x, Du, D^2 u) = f(x, u) \quad \mbox{in} \quad \Omega
$$
if for every $x_0 \in \Omega$ we have the following
\begin{enumerate}
  \item Either $\forall$ $\phi \in C^2(\Omega)$ such that $u-\phi$ has a local minimum at $x_0$ and $|D\phi(x_0)|\neq 0$ holds
      $$
      \mathrm{G}(x_0, D\phi(x_0), D^2 \phi(x_0)) \leq  f(x_0, \phi(x_0)) \quad (resp. \,\, \geq f(x_0, \phi(x_0)))
      $$
  \item Or there exists an open ball $B(x_0, \varepsilon)    \subset \Omega$, $\varepsilon >0$ where $u$ is constant, $u = K$ and holds
      $$
      f(x, K) \geq 0 \,\,\,\,\forall \,\,\,\, x \in B(x_0, \varepsilon) \quad (\mbox{resp.} \,\,f(x, K) \leq 0)
      $$
Finally, $u$ is said to be a viscosity solution if it is simultaneously a viscosity super-solution and
a viscosity sub-solution.
\end{enumerate}
\end{definition}

The following Harnack inequality will be an important tool for our arguments.

\begin{theorem}[{\bf Harnack inequality, \cite[Theorem 1.1]{DFQ2} and \cite[Theorem 1.3]{IS16}}] \label{harnack}
Let $u$ be a non-negative viscosity solution to
$$
   \mathrm{G}(x, Du, D^2u) = f \in C^{0}(B_1) \cap L^{\infty}(B_1).
$$
Then,
\begin{equation*}
	  \displaystyle \sup_{B_{\frac{1}{2}}} u(x) \leq C\left(N, \gamma, \lambda, \Lambda, \|\mathfrak{b}\|_{L^{\infty}(\Omega)}\right)\left( \inf_{B_{\frac{1}{2}}} u(x)+\|f\|_{L^{\infty}(B_1)}^{\frac{1}{\gamma+1}}\right).
	\end{equation*}
\end{theorem}

The next theorem plays a fundamental role in obtaining sharp and improved estimates along free boundary points of solutions. Such a result can be found in the references \cite[Theorem 3.1]{ART15}, \cite[Theorem 1.1]{BD2}, \cite[Theorem 1.1]{BD3}, \cite[Theorem 2]{C89}, \cite[Section 8.2]{CC95} and \cite[Theorem 1]{IS}.

In order to access such estimates, we will assume (additionally) the following assumption on $F$:
\begin{enumerate}
\item[{\bf(F3)}]\label{F3}[{\bf Continuity condition}] There exists a modulus of continuity $\omega, : [0, \infty) \to [0, \infty)$ with $\omega(0) = 0$ such that for all $(x, y, \overrightarrow{p}, M) \in \Omega \times \Omega \times (\R^N \setminus \{0\}) \times Sym(N)$
$$
      |F(x, \overrightarrow{p}, M)-F(y, \overrightarrow{p}, M)| \leq \omega(|x-y|)|\overrightarrow{p}|^{\gamma}\|M\|.
$$
\end{enumerate}

\begin{theorem}[{\bf Gradient estimates}]\label{GradThm} Let $u$ be a bounded viscosity solution to
$$
  \mathrm{G}(x, Du, D^2u) = f \in L^{\infty}(B_1).
$$
Then,
\begin{equation*}
	  \displaystyle \sup_{B_\frac{1}{2}} |Du(x)| \leq C\left(N, \gamma, \lambda, \Lambda, \|\mathfrak{b}\|_{L^{\infty}(\Omega)}\right)\left( \sup_{B_1} u(x)+ \|f\|_{L^{\infty}(B_1)}^{\frac{1}{\gamma+1}}\right).
	\end{equation*}
\end{theorem}

The next result is pivotal in order to prove the existence of viscosity solutions for our problem, as well as in proving some weak geometric properties soon. The proof holds the same lines as \cite[Lemma 3.2]{ALT} (see also \cite[Theorem 1.1]{BerDem} and \cite[Theorem 2.1]{BD2}). Hence, we will omit the proof here.

\begin{lemma}[{\bf Comparison Principle}]\label{comparison principle} Let $u_1$ and $u_2$ be continuous functions in $\overline{\Omega}$ and $f \in C^0([0, \infty))$ increasing with $f(0)=0$ fulfilling
$$
    \mathrm{G}(x, Du_1, D^2u_1)-\lambda_0(x)f(u_1) \leq 0 \leq \mathrm{G}(x, D u_2, D^2 u_2)-\lambda_0(x)f(u_{2}) \quad  \text{ in } \quad \Omega
$$
in the viscosity sense. If $u_1 \geq u_2$ on $\partial \Omega$, then $u_1 \geq u_2$ in $\Omega$.
\end{lemma}

Let us now comment on the existence of a viscosity solution of the Dirichlet problem \eqref{DCP1}. It follows by an application of Perron's method since a version of the Comparison Principle is available. In fact, let us consider functions $u^{\sharp}$ and $u_{\flat}$ that are solutions to the following boundary value problems:
\begin{equation}\nonumber
\begin{array}{ccc}
\left\{
\begin{array}{rcccc}
\mathrm{G}(x, D u^{\sharp}, D^2 u^{\sharp}) & = & 0 & \mbox{in} & \Omega, \\
u^{\sharp}(x) & = & g(x) &  \mbox{on} & \partial\Omega.\\
\end{array}
\right.
&
\mbox{and}
&
\left\{
\begin{array}{rllcc}
\mathrm{G}(x, D u_{\flat}, D^2 u_{\flat}) &=& \|g\|_{L^\infty(\partial\Omega)}^{\mu} & \mbox{in} & \Omega, \\
u_{\flat}(x) &=& g(x) &  \mbox{on} & \partial\Omega.\\
\end{array}
\right.
\\
\end{array}
\end{equation}
The existence of such solutions follows of standard arguments. Moreover, notice that $u^{\sharp}$ and $u_{\flat}$ are, respectively, super-solution and sub-solution to \eqref{DCP1}. Consequently, by Comparison Principle, Lemma \ref{comparison principle}, it is possible, under a direct application of Perron's method, to obtain the existence of a viscosity solution in $C(\overline{\Omega})$ to \eqref{DCP1}, more precisely we have the following theorem.

\begin{theorem}[{\bf Existence and uniqueness}]\label{ThmExist} Let $f \in C^0([0, \infty)) $ be a bounded, increasing real function with $f(0)=0$. Suppose that there exist a viscosity sub-solution $u_{\flat} \in C^0(\overline{\Omega}) \cap C^{0, 1}(\Omega)$ and a viscosity super-solution $u^{\sharp} \in C^0(\overline{\Omega}) \cap C^{0, 1}(\Omega)$ to $\mathrm{G}(x, Du, D^2 u) = f(u)$ satisfying
$u_{\flat} = u^{\sharp} = g \in C^^0(\partial \Omega)$. Define the class of functions
$$
     \mathrm{S}_{g}(\Omega) \defeq \left\{ v \in C^0(\overline{\Omega}) \;\middle|\; \begin{array}{c}
 v \text{ is a viscosity super-solution to } \\
\mathrm{G}(x, Du, D^2 u) = f(u) \text{ in } \Omega \text{ such that } u_{\flat} \le v \le u^{\sharp}\\
\text{ and } v = g \text{ on } \partial \Omega
\end{array}
\right\}.
$$
Then,
$$
   	u(x) \defeq \inf_{\mathrm{S}_{g}(\Omega)} v(x), \,\,\,\, \mbox{for} \,\, x \in \overline{\Omega}
$$
is (the unique) continuous (up to the boundary) viscosity solution to
$$
\left\{
\begin{array}{rclcc}
  \mathrm{G}(x, Du, D^2u) & = & f(u) & \mbox{in} & \Omega \\
  u(x) & = & g(x) & \mbox{on} & \partial \Omega.
\end{array}
\right.
$$
\end{theorem}

\section{Non-degeneracy result}\label{Geommeasprop}

\hspace{0.6cm}This Section is devoted to proving a geometric non-degeneracy property that plays an essential role in the description of solutions to free boundary problems of dead core type.

\begin{proof}[{\bf Proof of Theorem \ref{LGR}}]
Notice that, due to the continuity of solutions, it is sufficient to prove that such an estimate is satisfied just at point within $\{u>0\} \cap \Omega^{\prime}$ for $\Omega^{\prime} \Subset \Omega$.

First of all, for $x_0 \in \{u>0\} \cap \Omega^{\prime}$ let us define the scaled function $u_r(x) \defeq \frac{u(x_0+rx)}{r^{\frac{\gamma+2}{\gamma+1-\mu}}}$ for $x \in B_1$.

Now, let us introduce the comparison function: $\displaystyle \Xi (x) \defeq \left[\inf_{\Omega} \lambda_0 \frac{ \left(\gamma+1 - \mu \right)^{\gamma + 2} }{N \Lambda\left(\mu+1 \right) \left( \gamma+2\right)^{\gamma + 1}}\right]^{\frac{1}{\gamma+1 - \mu}} |x|^{\frac{\gamma+2}{\gamma+1-\mu}}$.

Straightforward calculus shows that
$$
     \mathcal{G}(x, D \Xi , D^2 \Xi) + |D\Xi|^{\gamma}\langle\mathfrak{b}_r(x), D \Xi\rangle- \hat{\lambda}_0\left( x  \right).\Xi^{\mu}(x) \leq 0 \quad \text{in} \quad B_1
$$
and
$$
\mathcal{G}(x, D u_r , D^2 u_r) +|D u_r|^{\gamma}\langle\mathfrak{b}_r(x), D u_r\rangle - \hat{\lambda}_0\left( x  \right). (u_r)_{+}^{\mu}(x) \geq 0  \quad \text{in} \quad B_1
$$
in the viscosity sense, where
$$
    \mathcal{G}(x, \overrightarrow{p} , M) \defeq r^{\frac{\gamma-2\mu}{\gamma+1-\mu}}F\left(x_0+rx, \overrightarrow{p}, r^{-\frac{\gamma-2\mu}{\gamma+1-\mu}}M\right), \,\,\,\mathfrak{b}_r(x) = r\mathfrak{b}(x_0 + rx)  \quad \mbox{and} \quad \hat{\lambda}_0(x) \defeq \lambda_0(x_0 + rx)
$$
Moreover, $\mathcal{G}$ satisfies the same structural assumptions than $F$, namely (F1)-(F2).

Finally, if $u_r \leq \Xi$ on the whole boundary of $B_1$, then the Comparison Principle (Lemma \ref{comparison principle}), would imply that
$$
   u_r \leq \Xi \quad \mbox{in} \quad B_1,
$$
which clearly contradicts the assumption that $u_r(0)>0$. Therefore, there is a point $Y \in \partial B_1$ such that
$$
      u_r(Y) > \Xi(Y) = \left[\inf_{\Omega} \lambda_0 \frac{ \left(\gamma+1 - \mu \right)^{\gamma + 2} }{N \Lambda\left(\mu+1 \right) \left( \gamma+2\right)^{\gamma + 1}}\right]^{\frac{1}{\gamma+1 - \mu}}
$$
and scaling back we finish the proof of the Theorem.
\end{proof}

\section{Improved regularity estimates}\label{ImpRegEst}

\hspace{0.6cm}In this Section we prove our improved regularity result along their free boundaries.

\subsection{A geometric iterative approach}

Before proving the main theorem this section, let us establish a key result of our approach.

\begin{definition}\label{FineClass} For a fully nonlinear operator $F$ fulfilling (F1)-(F2), $0\leq \mu< \gamma+1$ and $0< \mathrm{m}\leq \lambda_0 \leq \mathcal{M}$ we say that $u \in \mathfrak{J}(F, \lambda_0, \mathfrak{b}, \mu)(B_{2r_0}(x_0))$ if
\begin{enumerate}
  \item[\checkmark] $0 \leq u \leq 1\quad \mbox{in} \quad B_{2r_0}(x_0) \quad  \text{and} \quad \mathfrak{b} \in C^0(\overline{B_{2r_0}(x_0)}, \R^N)$ such that $\left\|\mathfrak{b}\right\|_{L^{\infty}(\overline{B_{2r_0}(x_0)})} \leq 1$.
  \item[\checkmark] $u(x_0)=0$.
    \item[\checkmark] $F(x, Du , D^2 u) + |D u|^{\gamma}\langle\mathfrak{b}(x), D u\rangle = \lambda_0(x)u^{\mu}_{+}(x) \quad \mbox{in} \quad B_{2r_0}(x_0)$ in the viscosity sense.
\end{enumerate}
\end{definition}

Next result regards the first step of a machinery of sharp geometric decay, which is a powerful device in nonlinear (geometric) regularity theory and plays a pivotal role in our approach. The core idea was inspired in \cite{LeeShah}, as well as in the flatness reasoning from \cite{Tei4}. Notwithstanding, the general class of operators which we are dealing with, it imposes some significant adjusts in such strategies. As a matter of fact, by invoking a Harnack type inequality for general fully nonlinear elliptic equations, and by showing that under a suitable control of \textit{Thiele modulus} (with small enough bounds) solutions fall into a flatness improvement regime (near their free boundaries) (compare with \cite[Lemma 3.1]{OSS}, \cite[Lemma 4.3]{daSRS18} and \cite[Lemma 4]{Tei18}).

\begin{lemma}[{\bf Flatness improvement regime}]\label{FlatLemma} Given $0<\iota<1$, there exists a $\delta_{\iota} = \delta(N, \lambda, \Lambda, \gamma, \iota)>0$ such that if $u \in \mathfrak{J}(F, \lambda_0, \mathfrak{b}, \mu)(B_{2r_0}(x_0))$ with $\left\|\lambda_0\right\|_{L^{\infty}(B_{2r_0}(x_0))} \leq \delta_{\iota}$ then
$$
    \sup_{B_{r_0}(x_0)} u(x) \leq 1-\iota.
$$
\end{lemma}

\begin{proof} Firstly, notice that
$$
   \inf_{B_{2r_0}(x_0)} u(x) = u(x_0) = 0 \quad \text{and} \quad \sup_{B_{2r_0}(x_0)} \lambda_0(u)^{\mu}_{+} \leq \delta_{\iota}.
$$
Now, by invoking the Harnack inequality (Theorem \ref{harnack}) we have that
$$
\begin{array}{rcl}
\displaystyle \sup_{B_{r_0}(x_0)} u(x)   & \leq  & \displaystyle C\left(N, \gamma, \lambda, \Lambda, \|\mathfrak{b}\|_{L^{\infty}(\overline{B_{2r_0}(x_0)})}\right)\left[ \inf_{B_{r_0}(x_0)} u(x)+\left(\sup_{B_{2r_0}(x_0)} \lambda_0(u)^{\mu}_{+}\right)^{\frac{1}{\gamma+1}}\right] \\
   & \le & C.\delta_{\iota}^{\frac{1}{\gamma+1}}.
\end{array}
$$
In particular, we conclude that
$$
   \displaystyle \sup_{B_{r_0}(x_0)} u(x) \leq 1-\iota
$$
provided $\delta_{\iota}< \left(\frac{1-\iota}{C}\right)^{\gamma+1}$. This finishes the proof.
\end{proof}

%\begin{remark} By revisiting the proof of the Lemma \ref{FlatLemma}, we conclude (in a direct way), from Harnack Inequality, that it is possible choice $\delta_{\iota}< \left(1-\iota\right)^{\gamma+1}$. Particularly, this makes clean, as well as highlighting the (quantitative) argument used in the proof.
%\end{remark}

\begin{remark}[{\bf Reducing to normalization, scaling and flatness condition}]\label{rem1} In order to prove Theorem \ref{IRThm} we need to argue that it is possible to reduce the assumptions, in a not restrictive way, to the one of Lemma \ref{FlatLemma}. For this reason, for any $x_0 \in \partial \{u>0\} \cap \Omega^{\prime}$ with $\Omega^{\prime} \Subset \Omega$, let us define
$$
     v_1(x)\defeq \frac{u\left( x_0 + \mathrm{R} x\right)}{\tau} \quad \mbox{in} \quad B_{2r_0} ,
$$
for $\tau, \mathrm{R}>0$ constants to be determined universally \textit{a posteriori}.

From the equation satisfied by $u$, we easily verify that $v$ fulfills in the viscosity sense
$$
  \mathcal{G}(x, Dv_1, D^2 v_1) + |D v_1|^{\gamma}\langle\hat{\mathfrak{b}}(x), D v_1\rangle = \hat{\lambda}_0(x).(v_1)^{\mu}_{+}(x),
$$
where
\begin{itemize}
  \item $\mathcal{G}(x, \overrightarrow{p}, M) \defeq \frac{\mathrm{R}^2}{\tau}F\left(x_0 + \mathrm{R} x, \overrightarrow{p}, \frac{\tau}{\mathrm{R}^2}M\right)$;
  \item $\hat{\mathfrak{b}}(x) = \mathrm{R}\mathfrak{b}(x_0 + \mathrm{R}x)$;
  \item $\hat{\lambda}_0(x) \defeq \frac{\mathrm{R}^{\gamma+2}}{\tau^{\gamma+1-\mu}}\lambda_0(x_0 + \mathrm{R} x)$,
\end{itemize}
with $\mathcal{G}$ fulfilling the assumptions (F1)-(F2).

Now, let $\kappa_{\lambda_0, \mu}>0$ be the greatest universal constant such the Lemma \ref{FlatLemma} holds provided
$$
   \left\|\mathcal{G}(x, Dv_1, D^2 v_1) + |D v_1|^{\gamma}\langle\hat{\mathfrak{b}}(x), D v_1\rangle\right\|_{L^{\infty}(B_1)} \leq \kappa_{\lambda_0, \mu}.
$$
Then, we make the following choices in the definition of $v$:
$$
   \tau \defeq \max\left\{1, \|u\|_{L^\infty(\Omega)}, \|\lambda_0\|_{L^{\infty}(\Omega)}^{\frac{1}{\gamma+1-\mu}}\right\} \,\,\, \mbox{and} \,\,\, \mathrm{R} \defeq \min\left\{2r_0, \frac{\text{dist}(\Omega^{\prime}, \partial \Omega)}{2}, \frac{1}{\|\mathfrak{b}\|_{L^\infty(\Omega)}+1}, \kappa_{\lambda_0, \mu}^{\frac{1}{\gamma+2}}\right\}
$$
Finally, with such selections $v_1$ fits into the framework of Lemma \ref{FlatLemma}. Consequently, we obtain
\begin{equation}\label{eqite}
       \displaystyle \sup_{B_{r_0}} v_1(x) \leq r_0^{\frac{\gamma+2}{\gamma+1-\mu}}.
\end{equation}
\end{remark}

Finally, we are in a position to supply the proof of Theorem \ref{IRThm}.

\begin{proof}[{\bf Proof of Theorem \ref{IRThm}}]
Let $u$ be a non-negative and bounded viscosity solution to \eqref{Maineq}, $x_0 \in \partial\{u > 0\} \cap \Omega^{\prime}$ any free boundary point and $r_0 \defeq \min\left\{\frac{1}{4}, \frac{\dist(\Omega^{\prime}, \partial \Omega)}{4}\right\}$. For the universal $\delta_{\iota}>0$ from Lemma \ref{FlatLemma} with $\iota \defeq 1-r_0^{\frac{\gamma+2}{\gamma+1-\mu}}$ we consider $v_1$ as in the Remark \ref{rem1}. Thus we define:
$$
     v_2(x) \defeq \frac{v_1(r_0 x)}{r_0^{\frac{\gamma+2}{\gamma+1-\mu}}} \quad \mbox{in} \quad B_{2}.
$$
Since $v_1 \in \mathfrak{J}(F, \lambda_0, \mathfrak{b}, \mu)(B_{2r_0}(x_0))$ and using \eqref{eqite}, it follows
\begin{enumerate}
    \item[\checkmark] $F(x, Dv_2 , D^2 v_2) + |D v_2|^{\gamma}\langle\hat{\mathfrak{b}}(x), D v_2\rangle = \hat{\lambda_0}(x)(v_2)^{\mu}_{+}(x) \quad \mbox{in} \quad B_{2}$ in the viscosity sense.

    \item[\checkmark] $0 \leq v_2 \leq 1\quad \mbox{in} \quad B_{2} \quad  \text{and} \quad \hat{\mathfrak{b}} \in C^0(\overline{B_{2}}, \R^N)$ such that $\left\|\hat{\mathfrak{b}}\right\|_{L^{\infty}(\overline{B_{2}})} \leq 1$ and $\|\hat{\lambda_0}\|_{L^{\infty}(\overline{B_{2}})} \leq \delta_{\iota}$.
  \item[\checkmark] $v_2(0)=0$.
\end{enumerate}
Hence, we are able to apply again the Lemma \ref{FlatLemma} to $v_2$ and obtain, after re-scaling for original domain,
$$
     \displaystyle \sup_{B_{r_0^2}} v_1(x) \leq r_0^{2\left(\frac{\gamma+2}{\gamma+1-\mu}\right)}.
$$
By iterating inductively as previously, we get the following geometric decay estimate:
\begin{equation}\label{eqInd}
    \displaystyle \sup_{B_{r_0^j}} v_1(x) \leq r_0^{j\left(\frac{\gamma+2}{\gamma+1-\mu}\right)}.
\end{equation}
In order to finish, for any fixed radius $0<r \leq \frac{\mathrm{R}}{2}$, let us choose $j \in \mathbb{N}$ such that
$$
     r_0^{j+1} < \frac{r}{\mathrm{R}} \leq r_0^{j}.
$$
For this reason, we can estimate
$$
     \displaystyle \sup_{B_r(x_0)} u(x) \leq \tau. \sup_{B_{r_0^j}(x_0)} v_1(x).
$$
Consequently, due to the estimate \eqref{eqInd} we can conclude that
$$
\begin{array}{ccl}
  \displaystyle \sup_{B_r(x_0)} u(x) & \leq & \left(\frac{1}{r_0\mathrm{R}}\right)^{\frac{\gamma+2}{\gamma+1-\mu}}\tau.r^{\frac{\gamma+2}{\gamma+1-\mu}}  \\
   & = & C\left(N, \lambda, \Lambda, \gamma, \mu, \dist(Y_0, \partial \Omega), \|\lambda_0\|_{L^{\infty}(\Omega)}\right)\max\left\{1, \|u\|_{L^\infty(\Omega)}, \|\lambda_0\|_{L^{\infty}(\Omega)}^{\frac{1}{\gamma+1-\mu}}\right\}.r^{\frac{\gamma+2}{\gamma+1-\mu}}.
\end{array}
$$
\end{proof}

\vspace{0.3cm}

As one of the consequences of our findings, we are able to establish a finer control for any viscosity solution to \eqref{Maineq} close its free boundary. Such a kind of information is crucial in a number of quantitative features for many free boundary problems (see \textit{e.g.} \cite{ART17}, \cite{daSRT}, \cite{RT} and \cite{Tei18}). Precisely, we prove that (near their free boundaries) solutions decay like an appropriated power of $\dist(\cdot, \partial \{u>0\})$.

\begin{corollary}\label{CorNonDeg}
Let $u$ be a non-negative, bounded viscosity solution to \eqref{Maineq} in $\Omega$. Given $x_0 \in \{u>0\} \cap \Omega^{\prime}$ with $\Omega^{\prime} \Subset \Omega$,
then
$$
  u(x_0) \leq C^{\sharp}
   \dist(x_0, \partial \{u>0\})^{\frac{\gamma+2}{\gamma+1-\mu}},
$$
where $C^{\sharp}>0$ is a universal constant\footnote{Throughout this manuscript, we will refer to {\it universal constants} when they depend only on the dimension and structural properties of the problem, i.e. on $N, \lambda, \Lambda, \gamma, \mu$ and the bounds of $\lambda_0$ and $\|\mathfrak{b}\|_{L^{\infty}(\Omega)}$}.
\end{corollary}

\begin{proof}
Fix $x_0 \in   \{u > 0\} \cap \Omega^{\prime}$ and denote $\mathrm{d}\defeq  \dist(x_0, \partial \{u>0\})$. Now, select $z_0 \in \partial \{u>0\}$ a free boundary point which achieves the distance, i.e., $\mathrm{d} = |x_0-z_0|$. From  Theorem \ref{IRThm} we have that
$$
  \displaystyle u(x_0) \leq \sup_{B_{\mathrm{d}}(x_0)} u(x) \leq \sup_{B_{2\mathrm{d}}(z_0)} u(x) \leq C^{\sharp}\left(N, \lambda, \Lambda, \gamma, \mu, \|\lambda_0\|_{L^{\infty}(\Omega)}, \|\mathfrak{b}\|_{L^{\infty}(\Omega)}\right)\mathrm{d}^{\frac{\gamma+2}{\gamma+1-\mu}},
$$
which finishes the proof.
\end{proof}

\subsection{The case $\mu = \gamma +1$: Proof of Theorem \ref{UCPthm}}

In the next we shall analyse the ``extreme case'' obtained as $\mu \to {\gamma +1}$, in other words,
\begin{equation}\label{UCP}
     F(x, Du, D^2 u) + |D u|^{\gamma}\langle\mathfrak{b}(x), D u\rangle = \lambda_0(x)u_{+}^{\gamma+1}(x) \quad \mbox{in} \quad \Omega.
\end{equation}

By means a barrier argument for the critical equation \eqref{UCP} and by using the Theorem \ref{IRThm} we shall prove that a non-negative solution to \eqref{UCP} can not vanish at an interior point, unless it is identically zero. This is a sort of strong maximum principle result.

\begin{proof}[{\bf Proof of Theorem \ref{UCPthm}}]
The prove will follow by \textit{reductio ad absurdum}. For that purpose, suppose that $\{u = 0\}\subsetneqq \Omega$, and lets $x_0 \in \Omega$ such that $u(x_0) > 0$. We can suppose without loss of generality that
$$
   \mathrm{d}_0 \defeq \dist(x_0,\partial\{u>0\}) < \frac{1}{10}\dist(x_0,\partial \Omega).
$$
We have that $u$ is locally bounded due to Comparison Principle (Lemma \ref{comparison principle}). Now, for fixed values of $\mathcal{A}>0$ and $s>0$ (large enough) we define the following barrier function:
$$
   \Theta_{\mathcal{A}, s}(x) \defeq \mathcal{A}\frac{e^{-s \frac{|x-x_0|^2}{\mathrm{d}_0^2}}-e^{-s}}{e^{-\frac{s}{4}}-e^{-s}},
$$
for which, a straightforward calculation shows (in the viscosity sense) that
\begin{equation}\nonumber
\left\{
\begin{array}{rclcl}
\mathcal{G}[\Theta_{\mathcal{A}, s}](x) & \geq & 0& \mbox{in} & B_{\mathrm{d}_0}(x_0) \setminus B_{\frac{\mathrm{d}_0}{2}}(x_0) \\
\Theta_{\mathcal{A}, s} & = & \mathcal{A} & \mbox{in} &  \partial B_{\frac{\mathrm{d}_0}{2}}(x_0)\\
 \Theta_{\mathcal{A}, s} & = & 0 & \mbox{in} & \partial B_{\mathrm{d}_0}(x_0),
\end{array}
\right.
\end{equation}
where $\mathcal{G}[u](x) = F(x, Du, D^2 u) + |D u|^{\gamma}\langle\mathfrak{b}(x), D u\rangle - \lambda_0(x)u^{\gamma+1}_{+}(x)$. Notice that

\begin{equation}\label{gradTheta}
\inf\limits_{B_{\mathrm{d}_0(x_0)} \setminus B_{\frac{\mathrm{d}_0}{2}}(x_0)}|D\, \Theta_{\mathcal{A}, s} (x) | \geq \frac{\mathcal{A} s \mathrm{d}_0^{-1}  e^{-s}}{e^{-\frac{s}{4}}-e^{-s}} \geq \frac{10\mathcal{A}s e^{-s}}{\dist(x_0,\partial \Omega)(e^{-\frac{s}{4}}-e^{-s})}\defeq \iota^{\ast} >0.
\end{equation}

On the other hand, for any constant $\zeta>0$,  the barrier $\zeta \Theta_{\mathcal{A}, s}$ still being a sub-solution in $B_{\mathrm{d}_0(x_0)} \setminus B_{\frac{\mathrm{d}_0}{2}}(x_0)$. In consequence, we have in the viscosity sense
$$
   \mathcal{G}[\zeta  \Theta_{\mathcal{A}, s}](x) \leq 0 \leq \mathcal{G}[u](x) \quad \text{in} \quad B_{\mathrm{d}_0(x_0)} \setminus B_{\frac{\mathrm{d}_0}{2}}(x_0).
$$
Moreover, by taking $\zeta_0 \in(0,1)$ small enough such that $\displaystyle \zeta_0 \mathcal{A} \leq \inf_{B_{\frac{\mathrm{d}_0}{2}}(x_0)} u(x)$ we obtain
$$
    \zeta_0 \Theta_{\mathcal{A}, s} \leq u \quad \mbox{in} \quad \partial B_{\mathrm{d}_0}(x_0) \cup \partial B_{\frac{\mathrm{d}_0}{2}}(x_0).
$$
Thus, by using Comparison Principle (Lemma \ref{comparison principle}) we obtain that
\begin{equation}\label{cprin}
    \zeta_0\Theta_{\mathcal{A}, s} \leq  u \quad \mbox{in} \quad B_{\mathrm{d}_0}(x_0) \setminus B_{\frac{\mathrm{d}_0}{2}}(x_0).
\end{equation}
On the other hand, for any $\max\{0, \gamma\}<\mu< \gamma+1$ fixed, we can rewrite equation \eqref{UCP} as
$$
      F(x, Du, D^2 u) + |D u|^{\gamma}\langle\mathfrak{b}(x), D u\rangle = h(x)u_{+}^{\mu}(x) \quad \mbox{in} \quad B_1,
$$
where $h(x)\defeq \lambda_0(x)u_{+}^{\gamma+1-\mu}(x)$. Thus, for $z \in \partial B_{\mathrm{d}_0} \cap \partial\{u>0\}$ we obtain by invoking Theorem \ref{IRThm} that
$$
      \displaystyle \sup_{B_r(z)} u(x)  \leq C\left(N, \lambda, \Lambda, \gamma, \mu, \|\mathfrak{b}\|_{L^{\infty}(B_1)}, \|h\|_{L^{\infty}(B_1)}\right).r^{\frac{\gamma+2}{\gamma+1-\mu}} \leq C.r^{\gamma+2}.
$$
for $r \ll 1$, since $1>\gamma+1-\mu$. Now, since $\gamma+2>1$ we can select $0<r_0 \ll 1$ small enough such that
$$
    	C.r_0^{\gamma+2} \le \dfrac{1}{7} \zeta.\iota^{\ast}.r_0.
$$
Finally, according to sentences \eqref{gradTheta} and \eqref{cprin} we obtain
$$
  \zeta.\iota^{\ast}.r_0 \leq  \sup\limits_{B_{r_0}(z)} \zeta \cdot |\Theta(|x|)-\Theta(|z|)| \leq  \sup\limits_{B_{r_0}(z)} \zeta. \Theta(|x|)  \leq  \sup\limits_{B_{r_0}(z)} u(x) \le  C.r_0^{\gamma+2} \le \dfrac{1}{7} \zeta.\iota^{\ast}.r_0,
$$
which yield a contradiction. Therefore, either $u>0$ or $u \equiv 0$ in $\Omega$.
\end{proof}

\begin{example} Theorem \ref{UCPthm} assures that non-trivial viscosity solutions to \eqref{Maineq} cannot dead-core sets, i.e., they must be strictly positive, provided $\mu = \gamma+1$. Indeed, fixed a direction $i = 1, \cdots, N$ and $\lambda_0>0$ we have for $\mathfrak{b}(x) = \sqrt[\gamma+2]{\lambda_0}$ that $u(x) = e^{\sqrt[\gamma+2]{\lambda_0}. x_i}$ is a strictly positive viscosity solution to
$$
    |D u(x)|^{\gamma} \left(\Delta u(x) + \langle\mathfrak{b}(x), Du(x)\rangle \right)= \lambda_0.u^{\gamma+1}(x) \quad \mbox{in} \quad B_1.
$$
\end{example}

\section{Consequences and further results}\label{Cons}

\hspace{0.6cm}Throughout this section we will present further consequences arising from our main results.

An important application of the Theorem \ref{IRThm}, we obtain a finer gradient control to solutions of \eqref{Maineq} near their free boundary points (cf. \cite[Corollary 4.1]{daSO}, \cite[Lemma 3.8]{OSS} and \cite[Lemma 3.8]{daSRS18}).

\begin{proof}[{\bf Proof of Theorem \ref{IRresult2}}] Firstly, let $x_0 \in \partial \{u>0\} \cap \Omega^{\prime}$ be an interior free boundary point. Now, we define the scaled auxiliary function $\Phi: B_1 \to \R_{+}$ by: $\Phi(x) \defeq \frac{u(x_0+rx)}{r^\frac{\gamma+2}{\gamma+1-\mu}}.$ Notice that $\Phi$ fulfils in the viscosity sense
$$
  \mathcal{G}(x, D \Phi, D^2 \Phi) + |D u(x)|^{\gamma}\langle\hat{\mathfrak{b}}(x), Du\rangle = \hat{\lambda}_{0}(x)\Phi^{\mu}(x) \quad \text{in} \quad B_1,
$$
where
$$
  \mathcal{G}(x, \overrightarrow{p}, M) \defeq r^{\frac{\gamma-2\mu}{\gamma+1-\mu}}F\left(z+rx, \overrightarrow{p}, r^{-\frac{\gamma-2\mu}{\gamma+1-\mu}}M\right), \,\,\,\hat{\mathfrak{b}}(x) = r\mathfrak{b}(z+rx) \quad \text{and} \quad \hat{\lambda}_0(x) \defeq \lambda_0(z + r x).
$$

Moreover, $\mathcal{G}$ satisfies the same structural assumptions than $F$, namely (F1)-(F3). From Theorem \ref{IRThm} we get that $\displaystyle \sup_{B_1} \Phi(x) \leq C$.

Finally, by invoking the gradient estimates (Theorem \ref{GradThm}) we obtain that
$$
\begin{array}{rcl}
  \displaystyle  \frac{1}{r^{\frac{1+\mu}{\gamma+1-\mu}}} \sup_{B_{\frac{r}{2}}(x_0)} |D u(x)| & = & \displaystyle \sup_{B_{\frac{1}{2}}(x_0)} |D \Phi(y)| \\
   & \leq & \displaystyle C\left(N, \gamma, \lambda, \Lambda, \|\mathfrak{b}\|_{L^{\infty}(B_1)}\right).\left[\sup_{B_1} \Phi(x) + \mathcal{M}^{\frac{1}{\gamma+1}}\left(\sup_{B_1} \Phi(x)\right)^{\frac{\mu}{\gamma+1}}\right] \\
   & \leq & \displaystyle C\left(N, \gamma, \lambda, \Lambda, \|\mathfrak{b}\|_{L^{\infty}(B_1)}\right)\max\left\{C, C^{\frac{\mu}{\gamma+1}}\right\}\left(1+ \mathcal{M}^{\frac{1}{\gamma+1}}\right).
\end{array}
$$
which finishes the proof.
\end{proof}

As a result, we obtain the following result in terms of distance up to the free boundary (see also \cite[Corollary 5.1]{daSS18}). The proof is similar to the one employed in Theorem \ref{CorNonDeg}.

\begin{corollary}Let $u$ be a bounded non-negative viscosity solution to \eqref{Maineq} in $B_1$. Then, for any point $z \in   \{u > 0\} \cap B_{\frac{1}{2}}$, there exists a universal constant $C>0$ such that
$$
  \displaystyle  |D u(z)| \leq  C\dist(z, \partial \{u>0\})^{\frac{1+\mu}{\gamma+1-\mu}},
$$
\end{corollary}

Now, we will obtain some measure theoretical properties of the phase-transition. Next result establishes that the positiveness region enjoys uniform positive density property along the free boundary. In particular, the development of cusps along free boundary points is inhibited.

\begin{corollary}[{\bf Uniform positive density}]\label{UPDFB}Let $u$ be a nonnegative, bounded viscosity solution to \eqref{Maineq} in $B_1$ and $x_0 \in \partial \{u > 0\} \cap B_{\frac{1}{2}}$ a free boundary point. Then for any $0<\rho< \frac{1}{2}$,
$$
     \Leb(B_{\rho}(x_0) \cap\{u>0\})\geq \theta.\rho^N,
$$
for a constant $\theta>0$ that depends only upon universal parameters and $\|u\|_{L^{\infty}}$.
\end{corollary}

\begin{proof}
Because of Theorem \ref{LGR}, for some $0<r<\frac{1}{2}$ fixed, it is possible to select a point $Y_0 \in  \overline{B_r(x_0)}$ such that,
\begin{equation}\label{dens}
    u(Y_0)= \sup\limits_{B_r(x_0)} u(x) \geq C_{\sharp}.r^{\frac{\gamma+2}{\gamma+1-\mu}}.
\end{equation}
In order to complete the proof the following inclusion
\begin{equation}\label{eq5.2}
    B_{\varsigma.r}(Y_0) \subset  \{u>0\}\cap B_1
\end{equation}
holds, for some $\varsigma>0$ (universal) small enough. In fact, from Theorem \ref{IRThm}, for $Z_0 \in \partial\{u>0\}$, we get
\begin{equation}\label{equap}
    u(Y_0) \leq C^{\sharp}.|Y_0 - Z_0|^{\frac{\gamma+2}{\gamma+1-\mu}}.
\end{equation}
Thus, by \eqref{dens} and \eqref{equap} we obtain
$$
     C_{\sharp}.r^{\frac{\gamma+2}{\gamma+1-\mu}} \leq C^{\sharp}.|Y_0-Z_0|^{\frac{\gamma+2}{\gamma+1-\mu}}
$$
and consequently,
$$
     \left(\dfrac{C_{\sharp}}{C^{\sharp}}\right)^{\frac{\gamma+1-\mu}{\gamma+2}}.r \leq |\,Y_0-Z_0|.
$$
For this reason, by taking $0< \varsigma \ll 1$ small enough, the inclusion in \eqref{eq5.2} is fulfilled. Therefore,
$$
     \Leb(B_{\rho}(x_0) \cap\{u>0\})\geq \Leb(B_{\rho}(X_0) \cap B_{\varsigma r}(Z_0))\geq \theta.r^N,
$$
\end{proof}

Next, we shall prove that the free boundary is a porous set. For the reader's convenience, we shall recall the definition of this notion.

\begin{definition}[{\bf Porous set}] A set $\mathrm{S} \subset \R^N$ is said to be porous with porosity constant $\epsilon \in (0, 1)$ if there is an $\mathrm{R}_0 > 0$ such that
$$
  \forall \,\,x \in \mathrm{S},\,\,\,\forall \,\,\, r \in \left(0, \mathrm{R}_0\right), \,\,\, \exists\, y \in \R^N\quad \text{such that} \quad B_{\epsilon r}(y) \subset B_r(x) \setminus \mathrm{S}.
$$
\end{definition}

\begin{corollary}[{\bf Porosity of the free boundary}]\label{CorPor} Let $u$ be a bounded, non-negative solution of \eqref{Maineq}. Then, there exists a universal constant $\xi>0$ such that
$$
     \mathcal{H}^{N-\xi}\left(\partial \{u>0\}\cap \Omega^{\prime}\right)< \infty.
$$
\end{corollary}

\begin{proof}
Let $R>0$ and $x_0\in\Omega$ be such that $\overline{B_{4R}(x_0)}\subset \Omega$. We will prove that $\partial \{u >0\} \cap B_R(x_0)$ is a $\frac{\epsilon}{2}$-porous set for some $\epsilon \in (0, 1]$. For this purpose, let $x\in \partial \{u >0\} \cap B_R(x_0)$. For each $r\in(0, R)$ we have $\overline{B_r(x)}\subset B_{2R}(x_0)\subset\Omega$. Now, let $y\in\partial B_r(x)$ such that $u(y)=\sup\limits_{\partial B_r(x)} u$. By Non-degeneracy (Theorem \ref{LGR})
\begin{equation}\label{5.1}
    u(y)\geq C_{\sharp}.r^{\frac{\gamma+2}{\gamma+1-\mu}}.
\end{equation}
On the other hand, near the free boundary (Corollary \ref{CorNonDeg})
\begin{equation}\label{5.2}
    u(y)\leq C^{\sharp}.\mathrm{d}(y)^{\frac{\gamma+2}{\gamma+1-\mu}},
\end{equation}
where $\mathrm{d}(y) \defeq \text{dist}(y, \partial \{u>0\} \cap \overline{B_{2R}(x_0)})$. Now, from \eqref{5.1} and \eqref{5.2} we get
\begin{equation}\label{5.3}
    \mathrm{d}(y)\geq\epsilon.r
\end{equation}
for a constant $0<\epsilon \defeq \left(\frac{C^{\sharp}}{C_{\sharp}}\right)^{\frac{\gamma+1-\mu}{\gamma+2}}\leq1$.

Consider now $\hat{y} \in[x,y]$ (the segment connection $x$ to $y$) such that $|y-\hat{y}|=\frac{\epsilon r}{2}$, then there holds
\begin{equation}\label{5.4}
   B_{\frac{\epsilon}{2}r}(\hat{y})\subset B_{\epsilon r}(y)\cap B_r(x).
\end{equation}
In effect, for each $z\in B_{\frac{\epsilon}{2}r}(\hat{y})$
$$
   |z-y|\leq |z-\hat{y}|+|y-\hat{y}|<\frac{\epsilon r}{2}+\frac{\epsilon r}{2}=\epsilon r,
$$
and
$$
   |z-x|\leq|z-\hat{y}|+\big(|x-y|-|\hat{y}-y|\big)\leq\frac{\epsilon r}{2}+\left(r-\frac{\epsilon r}{2}\right)=r,
$$
and \eqref{5.4} follows. Finally, since by \eqref{5.3} $B_{\epsilon r}(y)\subset B_{\mathrm{d}(y)}(y)\subset\{u>0\}$, then
$$
   B_{\epsilon r}(y)\cap B_r(x)\subset\{u>0\},
$$
which together with \eqref{5.4} implies
$$
   B_{\frac{\epsilon}{2}r}(\hat{y})\subset B_{\epsilon r}(y)\cap B_r(x)\subset B_r(x)\setminus\partial\{u>0\}\subset B_r(x)\setminus \partial\{u>0\} \cap B_R(x_0).
$$

Therefore, $\partial\{u>0\} \cap B_{R}(x_0)$ is a $\frac{\epsilon}{2}$-porous set. Finally, the desired $(N-\xi)$-Hausdorff measure estimate follows from \cite{KR}.
\end{proof}

\begin{remark}
 Particularly the Hausdorff estimate from Corollary \ref{CorPor} assures (see \cite{KR}) that
 $$
    \Leb(\partial\{u>0\} \cap \Omega^{\prime}) = 0.
 $$
\end{remark}

\section{A Liouville type result: Proof of Theorem \ref{Liouville}}\label{Sec6}

\hspace{0.6cm}The main purpose of this section is to prove that a global solution to
\begin{equation} \label{ecrn}
           F(x, Du, D^2 u) + |D u|^{\gamma}\langle\mathfrak{b}(x), Du\rangle =   \lambda_0(x) u_{+}^{\mu}(x) \quad \text{ in } \quad \R^N.
\end{equation}
must grow faster than $C|x|^{\frac{\gamma+2}{\gamma+1-\mu}}$ as $|x|\to \infty$ for a suitable constant $C>0$, unless it is identically zero.

For this end, fix $x_0 \in \R^N$, $\varsigma >0$ and $0<r_0<r$, we consider for $\rho<r$ (to be considered) the quantity $r_0=r-\rho$. Then, as in the proof of Theorem \ref{LGR}, it can be seen that the radially symmetric function $v: B_r(x_0) \to \R_+$ given by
$$
 \,v(x)=\Theta(N, \lambda_0, \gamma, \mu) (|x-x_0|-r_0)_+^{\,\frac{\gamma+2}{\gamma+1-\mu}}
$$
is a viscosity super-solution to
$$
   \left \{
       \begin{array}{rllll}
           F(x, Du, D^2  u) + |D u|^{\gamma}\langle\mathfrak{b}(x), Du\rangle & = & \lambda_0(x) u_{+}^{\mu}(x) & \text{ in } & B_r(x_0) \\
           u(x) & = & \varsigma &\text{ on } & \partial B_r(x_0)
\\
           u(x) & = & 0 &\text{ in } &  \overline{B_{r_0}(x_0)}
       \end{array},
   \right.
$$
where
$$
   \Theta(N, \Lambda, \lambda_0, \gamma, \mu) = \left[\inf_{\R^N}\lambda_0(x) \frac{ \left(\gamma+1 - \mu \right)^{\gamma + 2} }{N \Lambda\left(\mu+1 \right) \left( \gamma+2\right)^{\gamma + 1}}\right]^{\frac{1}{\gamma+1 - \mu}} \quad
\text{and} \quad  \quad \rho=\left( \frac{\varsigma}{\Theta(N,  \Lambda, \lambda_0, \gamma, \mu)} \right)^\frac{\gamma+1-\mu}{\gamma+2}.
$$

Such an explicit expression of $v$ allows us to prove our sharp (quantitative) Liouville type result.

\begin{proof}[{\bf Proof of Theorem \ref{Liouville}}]
Fixed $s_0>0$ (large enough), let us consider $w \colon \overline{B_{s_0}} \to \mathbb{R}$ the unique (see Theorem \ref{ThmExist}) viscosity solution to
\begin{equation}\label{ecc1}
\left\{
\begin{array}{rclcl}
  F(x, Dw, D^2 w) + |D \omega|^{\gamma}\langle\mathfrak{b}(x), D\omega\rangle & = & \lambda_0(x) w_{+}^{\mu}(x)& \text{ in } & B_{s_0}\\
  w(x) & = & \sup\limits_{\partial B_{s_0}} u(x) & \text{ on } & \partial B_{s_0}.
\end{array}
\right.
\end{equation}
According to the Comparison Principle (Lemma \ref{comparison principle}) $u \leq w \quad \mbox{in} \quad  B_{s_0}.$ Moreover, due to hypothesis \eqref{cond thm B1}
\begin{equation}\label{lim rad}
\sup\limits_{\partial B_{s_0}}\frac{u(x)}{s_0^{\frac{\gamma+2}{\gamma+1-\mu}}} \leq \sup\limits_{ B_{s_0}}\frac{u(x)}{s_0^{\frac{\gamma+2}{\gamma+1-\mu}}} \leq c\Theta(N,  \Lambda, \lambda_0, \gamma, \mu)
\end{equation}
for some $c \ll 1$ (small enough) and $s_0\gg 1$ (large enough). As above, the function
\begin{equation}\label{rad eq3}
  v(x)=\Theta(N,  \Lambda, \lambda_0, \gamma, \mu) \left(|x|-s_0 + \left(\frac{\sup\limits_{\partial B_{s_0}} u(x)}{\Theta(N,  \Lambda, \lambda_0, \gamma, \mu)}\right)^{\frac{\gamma+1-\mu}{\gamma+2}} \right)_+^{\frac{\gamma+2}{\gamma+1-\mu}}
\end{equation}
is a viscosity super-solution to \eqref{ecc1}. Thus, $w\leq v$ in $B_{s_0}$. Therefore, by  \eqref{lim rad} and \eqref{rad eq3} we conclude that
$$
   u(x) \leq \Theta(N,  \Lambda, \lambda_0, \gamma, \mu)  \left(|x|- (1-c^{\frac{\gamma+1-\mu}{\gamma+2}})s_0 \right)_+^{\frac{\gamma+2}{\gamma+1-\mu}} \to 0 \quad \mbox{as} \quad s_0 \to \infty.
$$
\end{proof}

\begin{example}
    It is worth highlighting that the constant in Theorem \ref{Liouville} is optimal in the sense that we can not remove the strict inequality in \eqref{cond thm B1}. In fact, the function given by
$$
    u(x) = \Theta(N,  \Lambda, \lambda_0, \gamma , \mu)(|x|-r_0)_{+}^{\frac{\gamma+2}{\gamma+1-\mu}}
$$
solves \eqref{ecrn} (for $\lambda_0(x) = \lambda_0$ a constant) and it clearly attains the equality in \eqref{cond thm B1} for the explicit value
$$
  \Theta(N,  \Lambda, \lambda_0, \gamma, \mu)  = \left[\lambda_0 \frac{ \left(\gamma+1 - \mu \right)^{\gamma + 2} }{N \Lambda\left(\mu+1 \right) \left( \gamma+2\right)^{\gamma + 1}}\right]^{\frac{1}{\gamma+1 - \mu}}.
$$
\end{example}

\subsection*{Acknowledgments}

\hspace{0.6cm}This work has been partially supported by CNPq (Brazilian government program \textit{Ci\^{e}ncia sem Fronteiras}), Consejo Nacional de Investigaciones Cient\'{i}ficas y T\'{e}cnicas (CONICET-Argentina), ANPCyT under grant PICT 2012-0153 and by PNPD-Capes (Universidade de Bras\'{i}lia - UnB). The authors would like to thank Eduardo V. Teixeira by his comments and suggestions that benefited a lot in the final outcome of this manuscript. J.V. da Silva and G.C. Ricarte would like to thank respectively \textit{Research Group on Partial Differential Equations} from Universidad de Buenos Aires and \textit{Analysis Research Group} of Centro de Matem\'{a}tica da Universidade de Coimbra for fostering a pleasant and productive scientific atmosphere during their Postdoctoral programs.

\end{document}